\documentclass[twoside,final]{amsart}

\usepackage{amsmath,amsfonts,calrsfs,amssymb,dsfont}
\usepackage{amsthm}
\usepackage[usenames, dvipsnames]{color}

\def\ps@myheadings{\let\@mkboth\@gobbletwo
  \def\@oddhead{{\slshape\rightmark}\hfil{\footnotesize\thepage}}
  \def\@oddfoot{}
  \def\@evenhead{{\footnotesize\thepage}\hfil\slshape\leftmark}
  \def\@evenfoot{}
  \def\sectionmark##1{}\def\subsectionmark##1{}
}

\oddsidemargin=\evensidemargin
\def\runninghead#1#2{\pagestyle{myheadings}
\markboth{{\protect\footnotesize\it{\quad #1}}}
{{\protect\footnotesize\it{#2\quad}}}}
\newtheorem{Theorem}{Theorem}[section]

\newtheorem{Proposition}[Theorem]{Proposition}
\newtheorem{Lemma}[Theorem]{Lemma}
\newtheorem{Corollary}[Theorem]{Corollary}
\newtheorem{Remark}[Theorem]{Remark}

\newtheorem{Hypothesis}{Hypothesis}
\makeatletter
\@addtoreset{equation}{section}

\makeatother

\date{August 12, 2016}

\newcommand{\eps}{\epsilon}
\newcommand{\R}{\mathbb R}
\newcommand{\bR}{\mathbb R}
\newcommand{\N}{\mathbb N}
\newcommand{\E}{\mathbb E}

\renewcommand{\P}{\mathbb P}

\newcommand{\bD}{\mathbb D}

\def\ds{\displaystyle}

\begin{document}

\setlength{\textheight}{7.7truein}    

\runninghead{S. Bonaccorsi, G. Da Prato and L. Tubaro}{Surface integrals and integration by parts formulae}

\thispagestyle{empty}
\setcounter{page}{1}

\centerline{\bf\large Construction of a surface integral under local Malliavin assumption}
\baselineskip=18pt
\par
\centerline{\bf \large and integration by parts formulae}

\baselineskip=13pt
\vspace*{0.37truein}

\centerline{Stefano BONACCORSI}
\baselineskip=12pt
\centerline{\footnotesize Dipartimento di Matematica, Universit\`a di
    Trento,}
\par
\centerline{\footnotesize via Sommarive 14, 38123 Povo (Trento), Italia}
\par
\centerline{\footnotesize \it stefano.bonaccorsi@unitn.it}

\baselineskip=18pt
\par

\centerline{Giuseppe DA PRATO}
\baselineskip=12pt
\centerline{\footnotesize Scuola Normale Superiore di Pisa,}
\par
\centerline{\footnotesize Piazza dei Cavalieri 7, 56126 Pisa, Italia}
\par
\centerline{\footnotesize \it daprato@sns.it}

\baselineskip=18pt
\par

\centerline{Luciano TUBARO}
\baselineskip=12pt
\centerline{\footnotesize Dipartimento di Matematica, Universit\`a di
    Trento,}
\par
\centerline{\footnotesize via Sommarive 14, 38123 Povo (Trento), Italia}
\par
\centerline{\footnotesize \it luciano.tubaro@unitn.it}

\vspace*{0.225truein}

\vspace*{0.21truein}

\centering{\begin{minipage}{4.5in}\footnotesize\baselineskip=10pt
    \parindent=0pt 
    {In this paper, we consider convex sets $K_r = \{g \ge r\}$ in an infinite dimensional Hilbert space, where $g$ is suitably
related to a reference Gaussian measure $\mu$ in $H$. We first show how
to define a surface measure on the level sets $\{g = r\}$ that is related to $\mu$. This allows to introduce an
integration-by-parts formula in $H$. This formula can be applied in several important constructions, as for instance
the case where $\mu$ is the law of a (Gaussian) stochastic process and $H$ is the space of its trajectories.}
    \par
    \parindent=15pt \footnotesize\baselineskip=10pt
    {\footnotesize\it Keywords}\/: 
    {Gaussian measures, surface measures in infinite dimensional spaces, integration-by-parts formulae, Neumann problem}
    \par
    \parindent=15pt \footnotesize\baselineskip=10pt
    {\footnotesize\it MSC 2010}\/: {60G15; 28C20}
    \end{minipage}}\par

{\tableofcontents }

\section{Introduction}

Let $H$ be a separable Hilbert space (with inner product denoted by $\langle \cdot,\cdot\rangle$ and norm $|\cdot|$) endowed with a 
non degenerate centered Gaussian measure $\mu=N_Q$;
we are given  
 a  continuous mapping $g:H\to \R$   and an open subset $I$ of $\R$ such that $I \subset g(H)$.

Our aim is to construct surface measures defined on level sets $\{g = r\}$ of the mapping $g$ under
local Malliavin conditions on $I$, see Hypothesis \ref{h1} below, and provide several integration-by-parts formulae.
Our construction extends previous results in the literature, and  in particular we shall rely on the procedure recently introduced
in \cite{DaLuTu14}:
under Hypothesis \ref{h1} below,  for each $r\in I$ there exists a  Borel  measure $\sigma_r$ in $H$, concentrated on   {the level surface} $\{g=r\}$  of $g$ such that for any $\varphi\in UC_b(H)$ we have
\begin{equation}
\label{e1.0}
F_\varphi(r):=\int_{\{g=r\}}\varphi\,{\rm d} \sigma_r=\lim_{\epsilon\to 0}\frac1{2\epsilon}\int_{\{r-\epsilon\le g\le r+\epsilon\}}\varphi\,{\rm d} \mu,\quad r\in I.
\end{equation}
We shall call $\sigma_r$ the {\em surface measure} related to $\mu$ on $\{g=r\}$.
\\
When $\sigma_r$ exists,  the measure  $(\varphi \mu)\circ g^{-1}$ is absolutely continuous with respect to  the  Lebesgue measure $\lambda$  restricted on $I$  and possesses a  continuous density which we denote by $\rho_\varphi(r),\,r\in I$. Moreover,   for any $r\in I$ there exists a   version of $\E^\mu[\varphi|g=r]$ such that
\begin{equation}
\label{e1.00}
\int_{\{g=r\}}\varphi\, {\rm d} \sigma_r=\E^\mu[\varphi|g=r]\;\rho_1(r), \quad \forall\;r\in I.
\end{equation}
\\
As by-product, we prove the following integration by parts formula
\begin{equation}
\label{IBP}
\int_{\{g=r\}}\langle Mg,z\rangle\, {\rm d} \sigma_r=-\int_{\{g\le r\}}( \langle M\varphi,z\rangle-\varphi\,\langle Q^{-1/2}x,z\rangle)\, {\rm d} \mu,
\end{equation}
under the assumptions $\varphi\in C^1_b(H)$,  $g\in \mathbb D^{1,2}(H,\mu)$, $z\in H,$ $\langle Mg,z\rangle\in UC_b(H)\cup\mathbb D^{1,2}(H,\mu)$  and $r\in I$, where $M$  is the Malliavin derivative and $\mathbb D^{1,2}(H,\mu)$ its domain (see below for further details).

Our basic assumption is the following
 
\begin{Hypothesis}
\label{h1}
There exist two random variables $u:H\to H$ and  $\gamma:H\to\R$ (both depending on $I$) such that
\begin{equation}
\label{e1.6}
 \langle Mg(x),u(x)  \rangle =\gamma(x),\quad \forall \,x\in g^{-1}(I)
\end{equation}
and
\begin{equation}
\label{e1.7}
\frac{u}{\gamma}\in D(M_p^*)\quad\forall\;p\ge 1.
\end{equation}
\end{Hypothesis}
Notice that if  $I=\R$, $u=Mg$ and  $\gamma=|Mg|^2$, Hypothesis \ref{h1} reduces to  the classical assumption
 from Airault--Malliavin \cite{AiMa88}, namely,
\begin{equation}
\label{e1.6a}
\frac{Mg}{|Mg|^2}\in D(M^*).
\end{equation}
Several papers have been devoted to the construction of  surface integrals under  assumption
\eqref{e1.6a} (requiring possibly some additional regularity on $g$), see for instance  \cite{AiMa88}, \cite{Bo98}, \cite{DaLuTu14} and the references therein. 

 It is well known, 
however, that \eqref{e1.6a} requires strong regularity for  the level surfaces $\{g=r\}$ and it is not fulfilled  for the function
\begin{equation*}
g(x)=\inf_{t\in[0,1]}x(t),\quad x\in L^2(0,1),
\end{equation*}
which arises in studying reflection problems on the set of positive functions, 
see \cite{NuPa92}, \cite{Za01}.

The {\em local Malliavin condition} provided by Hypothesis \ref{h1} is not new, 
since it has already  been introduced by D. Nualart, see \cite{Nu06} (with other   additional   regularity  assumptions both on $u$ and on $\gamma$)  but with a different purpose, namely for proving the existence of a density of  some Gaussian random variables with respect to the Lebesgue measure. 
As we shall see, this is   only a first step in  constructing  a surface measure.

\smallskip

In Section 3 we provide sufficient conditions on $g$ ensuring   Hypothesis \ref{h1}.
They require that $g$ has the special form
\begin{equation}
g(x)=\inf_{t\in[0,1]}X(t)(x),\quad x\in H,
\end{equation}
where  $X(t),\,t\in[0,1],$ is the solution of a stochastic differential equation on $\R$ with smooth coefficients. 
We shall need some results about  the unique  existence of a   minimum of $g(x)$
$\mu$--a.s. (Proposition \ref{p3.1}) and  a formula for the Malliavin derivative of $g$ (Proposition
\ref{p.2901}) which are probably known. We presents the  proofs, however, for the reader's convenience. 

We are able  to check  Hypothesis \ref{h1}   only  in few particular  situations:  more precisely, when:    i)  $X$ a real  Brownian motion, ii)  $X$ is a distorted   Brownian motion, iii)  $X$ is a geometric Brownian motion. 
We shall explain why we are not able so far to  handle more general cases, see Remark \ref{r3.7f}.

\smallskip

In  Section 4  we concentrate 
  on  integration by parts formulae 
 in the set $L^2_+(0,1)$ of all nonnegative functions. In this case  we provide more precise results than \eqref{IBP}.  The first results in this case where obtained by L. Zambotti, see \cite{Za01}.
 See also \cite{Ot09}.

    More precisely we   consider a Gaussian process $X$ in $[0,1]$ and its law $\mu=N_Q$ concentrated on $E=C([0,1])$ and set
\begin{align*}
g(x):=\min_{t\in [0,1]} X(t)(x),\quad x\in E
\end{align*}
assuming that the  law of $g$  is absolutely continuous with respect to  the Lebesgue measure  on $(-\infty,0]$ with a density   $\rho$ and that for  $\mu$--almost all  $x\in E$, $g(x)$  attains the minimum at  a unique   point   $\tau_x$, so that  for each  $z\in E$  we can write
  $Dg(x)\cdot z:=z(\tau_x)$, see Hypothesis \ref{h2} below. Under these conditions  we prove the integration by parts formula, equation \eqref{e6.2} below
\begin{multline*}
\E\big[z(\tau_x)\varphi\big | g\!=\!r\big]\rho(r)
\\
=-\int_{\{g\ge r\}}[D\varphi\cdot z- \langle Q^{-1/2}x, Q^{-1/2}z  \rangle\, \varphi]\, {\rm d} \mu,\;\forall\;r<0.
\end{multline*}
We  also study the limit case $r=0$  and apply the obtained results when   $X$ is respectively: (i)   a Brownian motion, (ii)  a distorted Brownian motion,  (iii) a Brownian Bridge. 
In case (i) and (iii) we  recover  by a different method the celebrate integration by parts formulae from   \cite{Za01} and   \cite{BoZa04}.

Finally, we shall give some concluding remarks   to some open problems, in particular   to the   Neumann problem on $\{g\ge r\}$ for the Kolmogorov operator
$$
\mathcal L\varphi=\frac12\,\mbox{\rm Tr}\,[D^2\varphi]-\frac12\,\langle    Q^{-1}x,D\varphi\rangle.
$$ \bigskip

We end this section with some notations. For a differentiable function $\varphi\colon H\mapsto \R$ we denote by $D$ the gradient operator in $H$. 
By $C_b(H)$ (resp. $UC_b(H)$) we mean the space of all real continuous (resp. uniformly continuous) and bounded mappings   $\varphi\colon H\to \R$ endowed with the sup norm $\|\cdot \|_{\infty}$. Moreover, $C^1_b(H)$   is  the subspace of $C_b(H)$  of all continuously differentiable functions, with bounded derivative.
\\
 It is well known  that the operator $M:=Q^{1/2}D$, defined in $C^1_b(H)$, is closable in $L^p(H,\mu)$ for all $p\in[1,+\infty)$; 
we denote by $M_p$ its closure,  by $\mathbb D^{1,p}(H,\mu)$ the domain of $M_p$ and by $M_p^*$  the adjoint  of $M_p$. We shall call $M_p$   the {\em Malliavin derivative} and $M^*_p$ the {\em Skorokhod integral}.  When no confusion may arise we shall omit the sub-index $p$.

\section{Constructing a surface measure}

\subsection{Differentiability of $F_\varphi$}
We assume here Hypothesis \ref{h1} and set
 \begin{equation}
\label{e1.1}
F_\varphi(r):=\int_{\{g\le r\}}\varphi\,{\rm d} \mu,\quad \forall\,r\in I,\;\forall\;\varphi\in L^1(H,\mu).
\end{equation}

We start by proving that  when  $\varphi\in \mathbb D^{1,p}(H,\mu)$ then $F_\varphi(\cdot)$ is continuously differentiable. This result  is   slightly different than a similar one from  \cite{Nu06}; we present a proof, however, both for   reader's convenience and   because the assumptions from \cite{Nu06} are somewhat stronger.
\begin{Lemma}
\label{l2.1}
Assume Hypotheses \ref{h1} and  take $\varphi\in \mathbb D^{1,p}(H,\mu)$ for some $p\ge 1$.  Then $F_\varphi$ is continuously differentiable at any    $r\in I$ and we have
 \begin{equation}
\label{e1.13}
F'_\varphi(r)=  \int_{\{g\ge r\}} M^*\left(\tfrac{u}{\gamma} \, \varphi \right)    \,{\rm d} \mu,
\end{equation}
\end{Lemma}

\begin{proof}
Fix $\epsilon_0 < {\rm dist}(r,I^c)$  \footnote{By $I^c$ we denote the complement of $I.$}; for every $\epsilon < \epsilon_0$ we write
\begin{align*}
\frac1\epsilon (F_\varphi(r+\epsilon)-F_\varphi(r)) = \frac1\epsilon \int_{H}{\mathds 1}_{\{r<g\le r+\epsilon\}}\,\varphi\,d\mu=\int_{H}h'_\epsilon(g)\,\varphi\,{\rm d} \mu,
\end{align*}
where
\begin{align*}
h_\epsilon(z) = \frac1\epsilon \int_{-\infty}^z {\mathds 1}_{[r,r+\epsilon]}(s) \, {\rm d} s.
\end{align*}
Since $h_\epsilon$ is Lipschitz continuous, we can apply the chain rule (formally, but the result can be obtained by approximating $h_\epsilon$ with functions in $C^1(\R)$)
\begin{equation}
\label{e1.8}
\varphi \, M(h_\epsilon(g)) = \varphi \, h_\epsilon'(g) Mg.
\end{equation}
Then, multiplying both sides of \eqref{e1.8} scalarly by $u$   it follows that
\begin{equation}
\label{e2.3c}
\varphi(x) \, \langle ((Mh_\epsilon(g))(x), u(x)    \rangle = \varphi(x) \, h'_\epsilon(g(x))\langle  (Mg)(x),u(x)  \rangle,\quad \forall \,x\in H.
\end{equation}
For $x \in g^{-1}(I)$, the right-hand side of previous equation is equal to 
\begin{align}\label{eq:s1416-1}
\varphi(x) \, h'_\epsilon(g(x)) \, \gamma(x)
\end{align}
thanks to Hypothesis \ref{h1}, while it is equal to
0 for every $x \not\in g^{-1}(I)$, by the definition of $h_\epsilon'$; however, since clearly also \eqref{eq:s1416-1} vanishes for $x \not\in g^{-1}(I)$,
we can use it on the whole $H$ to get
\begin{equation}
\label{e1.8b}
\varphi(x) \, h_\epsilon'(g(x)) = \varphi(x) \, \langle (M(h_\epsilon (g))(x), \tfrac{u(x)}{\gamma(x)}    \rangle,\quad \forall x\in H,\ \gamma(x) \not= 0.
\end{equation}
Integrating with respect to $\mu$ over $H$, yields
  \begin{equation}
\label{e1.11}
\int_H  \varphi(x) \, h_\epsilon'(g(x)) \, \mu(dx) =  \int_H  \varphi(x) \, \langle Mh_\epsilon(g(x)), \tfrac{u(x)}{\gamma(x)} \rangle \, \mu(dx).
\end{equation}


Notice that $\varphi \in \mathbb D^{1,p}(H)$ by assumption, while $\frac{u}{\gamma} \in D(M^*_q)$ for any $q \ge 1$ by Hypothesis \ref{h1},
therefore, H\"older's inequality implies that $\tfrac{u}{\gamma} \, \varphi \in D(M_q^*)$ for any $q > p$.
So, by duality and using Fubini's theorem, we have
\begin{align}
\frac1\epsilon (F_\varphi(r+\epsilon)&-F_\varphi(r)) = \int_H \varphi(x) \, h_\epsilon'(g(x)) \, \mu(dx) 
\nonumber \\
=&  \int_H \varphi(x) \, \langle Mh_\epsilon(g(x)), \tfrac{u(x)}{\gamma(x)} \rangle \, \mu(dx)
\nonumber \\
=& \int_H M^*\left(\varphi(x) \,\tfrac{u(x)}{\gamma(x)}\right) \, h_\epsilon(g(x)) \, \mu(dx)
\nonumber \\
=& \int_H \frac{1}{\epsilon} \int_{\R} M^*\left(\varphi(x) \,\tfrac{u(x)}{\gamma(x)}\right) \, \mathds 1_{[r,r+\epsilon]}(s) \mathds 1_{\{g(x) \ge s\}} \, ds \, \mu(dx)
\nonumber \\
=& \frac{1}{\epsilon} \int_r^{r+\epsilon} \int_H  M^*\left(\varphi(x) \,\tfrac{u(x)}{\gamma(x)}\right) \, \mathds 1_{\{g(x) \ge s\}} \, \mu(dx) \, {\rm d} s. \label{eq:s1416-2}
\end{align}
If we could apply the integral mean theorem to the right-hand side of \eqref{eq:s1416-2}, then
letting $\epsilon \to 0$ we obtain \eqref{e1.13}. This requires a little work.

Let $\varphi = 1$ and notice that \eqref{eq:s1416-2} implies the continuity of the mapping
\begin{align*}
s \mapsto F_1(s), \qquad s \in [r,r+\epsilon];
\end{align*}
however, by definition \eqref{e1.1}, $F_1(s) = \mu\{g \le s\}$, which implies the continuity of the integrand function in the right-hand side of \eqref{eq:s1416-2}
and, therefore, the proof is concluded.

By a similar argument, by using again H\"older's inequality, we can further prove continuity and even h\"olderianity of $F'_\varphi(\cdot)$.
\end{proof}

Let us introduce the following notation. For $\varphi \in \mathbb D^{1,p}(H,\mu)$, $p \ge 1$, we set
\begin{align*}
\rho_\varphi(r) = F'_\varphi(r) = \int_{\{g \ge r\}} M^*\left(\tfrac{u}{\gamma} \, \varphi\right) \, {\rm d} \mu, \qquad r \in I.
\end{align*}

 \begin{Remark}
\label{r1.1b}
\em 
 By Lemma \ref{l2.1} it follows that for all $\varphi\in \mathbb D^{1,p}(H,\mu)$ the measure   $(\varphi\mu)\circ g^{-1}$ is absolutely continuous with respect to the Lebesgue measure in $I$ having a continuous density $\rho_\varphi$.
\end{Remark}

Notice that, in particular, 
\begin{align*}
\rho_1(r) = \int_{\{g \ge r\}} M^*\left(\tfrac{u}{\gamma} \right) \, {\rm d} \mu, \qquad r \in I.
\end{align*}
By  \eqref{e1.13} it follows that

\begin{Corollary}
\label{c2.2c}
Let $\varphi\in \mathbb D^{1,p}(H,\mu)$, $r\in I$. Then there exists a constant $K>0$  and $q>p$ such that
\begin{equation}
\label{e2.7c}
|\rho_\varphi(r)|\le K\|\varphi\|_{\mathbb D^{1,q}(H,\mu)}.
\end{equation}
\end{Corollary}


The following lemma will be useful later.

\begin{Lemma}
\label{l2.3c}
Assume, besides Hypothesis \ref{h1}, that $\varphi\in \mathbb D^{1,p}(H,\mu)$. Then there is a continuous version of the function
\begin{equation*}
\E[\varphi|g=r]\,\rho_1(r),\quad r\in I 
\end{equation*}
such that
\begin{equation}
\label{e2.10b}
 \rho_\varphi(r) =\E[\varphi|g=r]\,\rho_1(r),\quad \forall\; r\in I.
\end{equation}
If in addition $\varphi\in UC_b(H)$  we have
\begin{equation}
\label{e2.10a}
|\rho_{\varphi}(r)|\le \|\varphi\|_\infty\,\rho_1(r),\qquad \forall r\in I,\;\forall \varphi\in \mathbb D^{1,2}(H,\mu).
\end{equation}
\end{Lemma}\begin{proof}
Write for $r\in I$
\begin{equation}
\label{e1.2b}
\begin{aligned}
\ds \rho_\varphi(r) &=\lim_{\epsilon\to 0}\frac1{2\epsilon}\int_H{\mathds 1}_{\{r-\epsilon\le g\le r+\epsilon\}}\,\varphi\,{\rm d} \mu\\
\ds &=\lim_{\epsilon\to 0}\frac1{2\epsilon}\int_H{\mathds 1}_{\{r-\epsilon\le g\le r+\epsilon\}}\,\E[\varphi|g]\,{\rm d} \mu\\
\ds &=\lim_{\epsilon\to 0}\frac1{2\epsilon}\int_{r-\epsilon}^{r+\epsilon}\,\E[\varphi|g=s]\,\rho_1(s)ds=\E[\varphi|g=r]\,\rho_1(r),\quad r\mbox{\rm--a.s. in}\;I.
\end{aligned}
\end{equation}
Since $\rho_\varphi(\cdot)$ is continuous, the first 
statement of the Lemma follow from Lemma \ref{l2.1}.

To show \eqref{e2.10a} write
\begin{equation*}
|\E[\varphi|g]\le \E[|\varphi||g]\le \|\varphi\|_\infty.
\end{equation*}
Therefore the conclusion follows.
\end{proof}

\subsection{The surface measure}

Now we are going to  prove the main result of this section, namely that   the positive functional $r\to \rho_\varphi(r)$ is  in fact the integral of $\varphi$  with respect to a given (surface) measure $\sigma_r$. The key point is to extend the functional to  all $\varphi\in UC_b(H)$ as the following proposition shows.
\begin{Proposition}
 \label{p5.6}
  For any $\varphi\in UC_b(H)$, $F_{\varphi}$ is continuously differentiable on $I$.
 \end{Proposition}
\begin{proof}
If $\varphi\in \mathbb  D^{1,p}(H,\mu)$ the {result follows from Lemma \ref{l2.1}}.  Now let   $\varphi\in UC_b(H)$ and let $(\varphi_n)$ be a sequence in $UC_b^1(H)$  convergent to $\varphi$ in $UC_b(H)$. Since $\varphi_n\in  D^{1,p}(H,\mu)$ by \eqref{e2.10a} we have
\[
|\rho_{\varphi_n}(r)-\rho_{\varphi_m}(r)|=|\rho_{\varphi_n-\varphi_m}(r)|\le \|\varphi_n-\varphi_m\|_\infty\;\rho_1(r),\quad\forall\;r\in I,
\]
hence $\{\rho_{\varphi_n}\}$ is a Cauchy sequence in $UC_b(H)$ and the conclusion follows.

 \end{proof}
 \begin{Remark}
 \label{r2.5c}
 \em Assume that  $M^*(\tfrac{u}\gamma)\in UC_b(H)\cup \mathbb D^{1,2}(H,\mu).$
 Then $F_{\varphi}$ is twice continuously differentiable on $I$.
 
 \end{Remark}
Now we we can prove 
\begin{Theorem}
\label{t5.7}
Assume  Hypothesis \ref{h1}. Then for every $r\in I$ there exists a unique   Borel measure $\sigma_r$ on $H$  such that 
\begin{equation}
\label{e5.11}
F_\varphi '( r)=\rho_\varphi(r)  =\int_{H} \varphi(x)\,\sigma_r(dx),\quad\forall\; \varphi\in  UC_b(H). 
\end{equation}
Moreover, $\sigma_r(H)=\rho_1(r)$ and if $g$ is continuous the support of $\sigma_r$ is included in $\{g=r\}$.   
\end{Theorem}  

\begin{proof}
 The proof  is  similar 
 to that of \cite[Theorem 3.5]{DaLuTu14}, so it will be only sketched. {Let us fix $r\in I$. By Proposition \ref{p5.6} the functional
\begin{equation*}
 UC_b(H)\to \R,\quad \varphi\to \rho_r(\varphi),
\end{equation*}
 is well defined and   clearly positive. To show that it is a measure we follow} a classical method.
First  we construct  a suitable increasing sequence $(K_n)$ of compact sets  of $H$ converging to $H$. Then we introduce  the restrictions    {$\rho^n_\varphi(r)$ of  $\rho_\varphi(r)$ to $K_n$   for all  $\varphi\ge 0$ setting
\begin{equation*}
\rho^n_\varphi(r)=\inf\Big\{\rho_\psi(r):\;\psi\in  UC_b(H),\;\psi= \varphi\;\mbox{\rm on}\;K_n,\;\;\psi\ge 0\;\mbox{\rm on}\;H \Big\},
\end{equation*}
while if $\varphi $ takes  both positive and negative values,  $\rho^n_\varphi(r)$ is defined by
\begin{equation*}
\rho^n_\varphi(r) = (\rho^n_\varphi(r))^+ - (\rho^n_\varphi(r))^-, 
\end{equation*}
where  $\varphi^+ $ and $ \varphi^-$ denote the positive and the negative part of  $\varphi$.
  Then  $\rho_\varphi(r)$ is a positive linear functional in  $C(K_n)$ as easily checked. Now for each $r\in I$  $\varphi\to \rho^n_\varphi(r)$}  are measures in view of the Riesz representation theorem.
  Finally, it is not difficult to show that  {$\rho^n_1(r)\uparrow \rho_1(r),$ which implies that $\rho_\varphi(r)$ is a measure as well. }
\end{proof}

In the applications it is important to know whether  a Borel function $\varphi:H\to \R$ (not necessarily belonging to $UC_b(H)$)  has a {\em trace} on the surface $\{g=r\}$ for some $r\in I$. When the Malliavin condition \eqref{e1.6a} is fulfilled this problem was investigated in \cite{CeLu14}, see also \cite{DaLuTu14}. Given $r\in I$, we shall say that  $\varphi$ possesses a trace $T\varphi$ on  $\{g=r\}$ if there exists a sequence $(\varphi_n)\subset UC_b(H)$ such that
$$
\varphi_n\to T\varphi,\quad\mbox{\rm on}\; L^1(H,\mathcal B(H), \sigma_r).
$$
\begin{Proposition}
\label{p1.2c}
Let   $\varphi\in \mathbb D^{1,2}(H,\mu)$  and let   $(\varphi_n)\subset C^1_b(H)$  be a sequence convergent   to  $\varphi$ in $ \mathbb D^{1,2}(H,\mu)$. Then    $(\varphi_n)$ is Cauchy in $L^1(H,\sigma_r)$, so that
$\varphi$ possesses a trace $T\varphi$ on  $\{g=r\}$.
\end{Proposition}

\begin{proof}

We first notice that by passing if necessary to an approximating sequence in $UC_b^1(H)$, and using Lemma \ref{l2.3c} it follows that
\begin{equation}
\label{e1.10d}
F'_\varphi(r)=\int_{\{g=r\}}\varphi\,{\rm d} \sigma_r
=\E[\varphi|g=r]\,\rho_1(r),\quad \forall\,\varphi\in UC_b(H),\;r\in I.
\end{equation}
Let now $\varphi\in \mathbb D^{1,2}(H,\mu)$  and let   $(\varphi_n)\subset C^1_b(H)$  be a sequence convergent   to  $\varphi$ in $ \mathbb D^{1,2}(H,\mu)$. {We claim} that  $(\varphi_n)$ is Cauchy in $L^1(H,\sigma_r)$. In fact
\[
\int_H |\varphi_n-\varphi_m|\,{\rm d} \sigma_r=\rho_{ |\varphi_n-\varphi_m|}(r)\le
K(r)\;\| \varphi_n-\varphi_m\|_{\mathbb D^{1,2}(H,\mu)},
\]
thanks to Corollary \ref{c2.2c}. 
\end{proof}

\section{Fulfilling Hypothesis \ref{h1}}

\subsection{The maximum of a stochastic flow}
Let us  start with some general results concerning a stochastic differential equation on $\R$
\begin{equation}
\label{e5}
{\rm d}X=b(X) \, {\rm d}t+\sigma(X) \, {\rm d}B(t),\quad X(0)=\xi\in\R,
\end{equation}
where $b$ and $\sigma$ are of class $C^2$ and Lipschitz continuous  and
 $B(\cdot)$ is a Brownian motion on $(H,\mathcal B(H),\mu)$.
 We denote by $X(t,\xi) =X(t)$ the strong solution of \eqref{e5}.

 We are going to consider the function on $H$
\begin{equation}
\label{e7}
g(x)=\sup_{t\in[0,1]} X(t)(x),\quad x\in H.
\end{equation}
Since the trajectories of $X(t)$ are $\mu$--a.s. continuous the supremum in \eqref{e7} is indeed a maximum $\mu$--a.s.

\begin{Proposition}
\label{p3.1}
Let $X$ be the solution to   \eqref{e5} and   assume that  the joint probability distributions of any order admit a density with respect to the Lebesgue measure. Then $X(t)(x),\;t\in[0,1],$ attains the maximum at a unique point $\tau_x$ of $[0,1]$  for $\mu$--almost $x\in H$.
\end{Proposition}

\begin{proof} 
We proceed by steps. First, we consider two time points $0\le s<t$: we have that
\begin{align*}
{\mu(X(s) =X(t))} = \int_{\mathbb R} \int_{\mathbb R} {\mathds 1}_{\{x=y\}} f_{(X(t), X(s))}(x,y) \,{\rm d} x\, {\rm d} y = 0,
\end{align*}
{where $ f_{(X(t), X(s))}$ is the joint density of $(X(t),\,X(s))$}.

Next, we extend the analysis to three times.
Consider $s_1<s_2<t$ and the corresponding random variables   $g_{(s)}=\max \{X(s_i),i=1,2\}$ and $X_t$, we have
\begin{align*}
{\mu(g_{(s)}\le \xi, X(t) \le \eta)} =&
{\mu(X(s_1)\le \xi,X(s_2)\le \xi,X(t)\le \eta)}
\\
=& \int_{-\infty}^\xi\int_{-\infty}^\xi\int_{-\infty}^\eta f(x_1,x_2,y) \,{\rm d} x_1\, {\rm d} x_2 \,{\rm d} y,
\end{align*}
that has a density given by
\begin{align*}
\int_{-\infty}^\xi f(x_1,\xi,\eta)\,  {\rm d} x_1+\int_{-\infty}^\xi f(\xi,x_2,\eta)\,  {\rm d} x_2.
\end{align*}
Considering now $s_1<s_2<t_1<t_2$ and the corresponding random variables   $g_{(s)}=\max \{X(s_i),i=1,2\}$ and $g_{(t)}=\max \{X(t_i),i=1,2\}$, we have
\begin{align*}
{\mu(g_{(s)}\le \xi, g_{(t)}\le \eta) }&=
{(X(s_1)\le \xi,X(s_2)\le \xi,X(t_1)\le \eta),X(t_2)\le \eta)}
\\
=& \int_{-\infty}^\xi\int_{-\infty}^\xi\int_{-\infty}^\eta\int_{-\infty}^\eta f(x_1,x_2,y_1,y_2)\, {\rm d} x_1\, {\rm d} x_2 \, {\rm d} y_1 \, {\rm d} y_2
\end{align*}
that has a density given by
\begin{multline*}
\int_{-\infty}^\xi \int_{-\infty}^\eta f(x_1,\xi,y_1,\eta)\,{\rm d} x_1 \,  {\rm d} y_1
+\int_{-\infty}^\xi \int_{-\infty}^\eta f(x_1,\xi,\eta,y_2)\,  {\rm d} x_1 \,   {\rm d} y_2
\\
+\int_{-\infty}^\xi \int_{-\infty}^\eta f(\xi,x_2,y_1,\eta)\,   {\rm d} x_2 \,   {\rm d} y_1
+\int_{-\infty}^\xi \int_{-\infty}^\eta f(\xi,x_2,\eta,y_2)\,  {\rm d} x_2 \,{\rm d} y_2.
\end{multline*}

Let us consider a system of points \[s_1<s_2<\cdots <s_m<t_1<t_2<\cdots<t_n\] and define
$g_{(s)}=\max \{X(s_i),i=1,2,\ldots,m\}$ and $g_{(t)}=\max \{X(t_i),i=1,2,\ldots,n\}$: we claim that
\begin{equation}\label{etub1}
{\mu(g_{(s)}=g_{(t)})}=0.
\end{equation}
In fact, the joint probability distribution of $(g_{(s)},g_{(t)})$ (using the cumulative distribution function) given by
$$
\begin{array}{l}
\mu(g_{(s)}\le \xi, g_{(t)}\le \eta)=\\
\\
\mu(X(s_1)\le \xi,X(s_2)\le \xi,\ldots,X(s_m)\le \xi,X(t_1)\le \eta,X(t_2) \le \eta,\ldots,X(t_n)\le \eta)\\
\\
\ds
= \int_{-\infty}^\xi
\int_{-\infty}^\xi
\cdots
\int_{-\infty}^\xi
\int_{-\infty}^\eta
\int_{-\infty}^\eta
\cdots
\int_{-\infty}^\eta
f(\mathbf x,\mathbf y)\,d\mathbf x \, d\mathbf y,
\end{array}
$$
admits a density with respect to the Lebesgue measure given by
\[
\int_{-\infty}^\xi \int_{-\infty}^\eta\Big[\sum_{i,j} f(\mathbf x_i,\mathbf y_j)\Big] \, d x \,d y
\]
where $$\mathbf x_i=(x_1,x_2,\ldots,x_{i-1},x,x_{i+1},\ldots,x_m)$$ and analogously
$$\mathbf y_j=(y_1,y_2,\ldots,y_{j-1},y,y_{j+1},\ldots,y_n),$$
from which we get \eqref{etub1}.
\\
Now, given two disjoint intervals $I=[a_1,b_1]$ and $J=[a_2,b_2]$ (with $b_1<a_2$) in $[0,T]$,
we get, by the continuity of paths,
\[
\mu(\max_{s\in I}X(s)=\max_{t\in J}X(t))= \mu(\sup_{s\in\mathbb Q\cap I}X(s)=\sup_{t\in\mathbb Q\cap J}X(t))
\]
where $\mathbb Q$ is the set of rational numbers. On the other hand let $\mathbb Q_n$ be an increasing sequence of finite number of rationals
such that $\bigcup \mathbb Q_n=\mathbb Q$; we have
\[
 \mu\left(\sup_{s\in\mathbb Q\cap I}X(s)=\sup_{t\in\mathbb Q\cap J}X(t)\right)\le  \mu\left(\sup_{s\in\mathbb Q_n\cap I}X(s)=\sup_{t\in\mathbb Q_n\cap J}X(t)\right)=0.
\]
Therefore, the proposition is proved.
\end{proof}

 Let us compute the Malliavin derivative of $g$.
  \begin{Proposition}
 \label{p.2901}
 Function  $g$ defined by \eqref{e7} belongs
to $  \bD^{1,2}(H,\mu)$ and it results
\begin{align}
\label{e4.5c}
M g(x) = MX(t)\Big|_{t=\tau_x} \quad\quad \mu\mbox{\rm --a.s. in}\;H. 
\end{align}
\end{Proposition}

\begin{proof}
The first part follows by the general criterium stated in \cite[Proposition 2.1.10]{Nu06},
which holds for general continuous processes $X$.
\\
According to such result, in order to establish \eqref{e4.5c} we 
fix a countable and dense set $\{t_n,\ n \ge1\}$ in $[0,1]$; notice that 
we can approximate $g$ with a sequence of discrete random variables
$g_n := \max\{X(t_1), \dots, X(t_n)\}$. Let $\phi_n(x_1, \dots, x_n)
= \max\{x_1, \dots, x_n\}$; then $\phi_n$ is a Lipschitz continuous function
with 
\begin{align*}
M \phi_n(X(t_1), \dots, X(t_n)) = M X(\tau_n), 
\end{align*}
where $\tau_n\in \{t_1,t_2,\ldots,t_n\}$ is the time when $X(\tau_n) = \max\{X(t_1), \dots, X(t_n)\}$.
Passing to the limit as $n \to \infty$, due to the continuity of the trajectories, the claim follows.
\end{proof}

\subsection{Applications}

We start  with the maximum of the Brownian motion. Set   
\begin{equation}
\label{e9}
S(t)=\max_{s\in[0,t]}B(s),\quad\forall\;t\in[0,1].
\end{equation}
By Proposition \ref{p3.1} for almost  every $x\in H$ and any $t\in[0,1]$, $S(\cdot)(x)$ attains the maximum on $[0,t]$ at a unique point denoted  $\tau_x^t$.
Moreover,  by Proposition \ref{p.2901} for any $t\in[0,1]$, $S(t)\in \mathbb D^{1,2}(H,\mu)$ and the Malliavin derivative of $S(t)$ is given by
\begin{equation}
\label{e10}
MS(t)(x)={\mathds 1}_{[0,\tau_x^t]}.
\end{equation}

Let us fix $a>0$  and set $I=(a,+\infty)$.
Our aim is to show  the following result.
\begin{Proposition}
\label{p4}
The function
\begin{equation*}
g(x)=S(1)(x),\quad x\in H,
\end{equation*}
fulfills the local Malliavin  condition on $I$  with
\begin{equation}
\label{e11}
u_t(x)=\psi(S(t)(x)),\quad \gamma(x)=\int_0^1\psi(S(t)(x)) \, {\rm d} t,
\end{equation}
where   $\psi:[0,+\infty)\to [0,+\infty)$ is  $C^\infty$ and such that
\begin{equation}
\label{e11bis}
\psi(r)=
\begin{cases}
1\quad&\mbox{\rm if}\;r\in[0,a/2]\\
\\
0\quad&\mbox{\rm if}\;r\ge a.
\end{cases}
\end{equation}
\end{Proposition}
\begin{proof}
We have to show that
\begin{equation}
\label{e12}
\int_0^1 M_tg(x)\,\psi(S(t)x)\,{\rm d} t=\gamma(x),\quad \forall\;x\in g^{-1}(I)
\end{equation}
and
\begin{equation}
\label{e13}
\frac{u}{\gamma}\in D(M^*).
\end{equation}

We proceed in several steps.\medskip

\noindent{\it Step 1}. Identity \eqref{e12} is fulfilled.\medskip

Let $x$ be fixed  such that $g(x)>a$ and write, taking into account \eqref{e10}
\begin{equation*}
\int_0^1 M_tg(x)\,\psi(S(t)x) \,{\rm d}t = \int_0^{\tau_x^1} \psi(S(t)x)\,{\rm d} t.
\end{equation*}
Function $t\to S(t)x$ is increasing, so there is
$t_x<\tau_x^1$ such that  $S(t_x)x=a$. Therefore (since $\psi(r)=0$ for $r\ge a$)
\begin{equation*}
\int_0^1 M_tg(x)\,\psi(S(t)x)\,{\rm d} t=\int_0^{t_x} \psi(S(t)x)\,{\rm d} t=\int_0^1\psi(S(t)x)\,{\rm d} t=\gamma(x).
\end{equation*}

\noindent{\it Step 2}. $\gamma\in \mathbb D^{1,p}(H,\mu)$ for all $p\ge 1$.
\\
By the chain  rule and \eqref{e10} we have that $\gamma\in  \mathbb D^{1,p}(H,\mu)$ and
$$
M\gamma(x)=\int_0^1\psi'(S(t)x)\,{\mathds 1}_{[0,\tau_x^t]}\,{\rm d} t.
$$

\medskip

\noindent{\it Step 3}.  $u\in D(M_p^*)$ for all $p\ge 1$.
\\
Actually, $u$ is an adapted process, $u \le 1$, hence $u$ is It\^o integrable and,
{\it a fortiori}, Skorohod integrable.

\medskip


\noindent{\it Step 4}. $\frac1\gamma\in L^p(H,\mu)$ for all $p\ge 1$. 
\\
First notice that (recall that $S(\cdot)x$ is not decreasing)
\begin{multline*}
\gamma(x)= \ds \int_0^1\psi(S(t)x)\,{\rm d} t\ge (S(\cdot)x)^{-1}(\tfrac{a}2)
\\
=\ds (S(.)x)^{-1}(r)=\inf\{ s\ge 0 :   S(s)x\ge r \}.
\end{multline*}
Therefore
\begin{align*}
\frac1{\gamma(x)}\le \frac1{(S(\cdot)x)^{-1}(\tfrac{a}2)}=:Z(x).
\end{align*}
 So, it is enough to show that $Z\in L^p(H,\mu)$, equivalently that
 \begin{equation}
\label{e14}
\int_H Z^p\, {\rm d}\mu=p\int_0^\infty \mu(Z>\epsilon)\,\epsilon ^{p-1}\,{\rm d} \epsilon<\infty.
\end{equation}
Notice now that for any $q>1$ we have
$$
\{Z>\epsilon\}=\left\{\frac{a}2< S(\tfrac1\epsilon)x\right\}\le \left(\frac2a\right)^q\int_H\sup_{s\in [0,1/\epsilon]} B(s)^q   {\rm d} \mu\le c_q\left( \tfrac1\epsilon  \right) ^{q/2}.
$$
So the conclusion follows from the arbitrariness of $q$. 

\medskip

\noindent{\it Step 5}.    $\frac{u}\gamma\in D(M^*)$. 
\\
Let us first notice that $\gamma^{-1}\in D^p(H,\mu)$ since by the chain rule
\begin{align*}
M_p(\gamma^{-1})=-\gamma^{-2}M_p\gamma.
\end{align*}
Therefore $\frac{u}\gamma\in D(M^*)$ and by Step 3 we have
\begin{align*}
M_p^*(\tfrac{u}\gamma)=\gamma^{-1}\,M^*_pu-\langle M_p(\gamma^{-1}),u\rangle.
\end{align*}
The proof is complete.
 \end{proof}

\begin{Remark}
 \em The proof above  was inspired by the paper  from \cite{FlNu95} about the supremum of the Brownian sheet, but it is  more elementary. In particular,  it does not require fractional Sobolev spaces and  the Garsia, Rodemich and 
Rumsey result as in the quoted paper. 
 \end{Remark}
 
 Now we consider a {\em distorted Brownian motion}
\begin{equation}
\label{e15}
B_{b,\sigma}(t):= b\,t+\sigma B(t),\quad\forall\;t\in[0,1],
\end{equation}
where $b\in\R$ and $\sigma>0$ are given. Set
\begin{equation}
\label{e15f}
S_{b,\sigma}(t):=\sup_{s\in[0,t]}(bs+\sigma B(s)),\quad\forall\;t\in[0,1].
\end{equation}

Again, by Proposition \ref{p3.1}  $S_{b,\sigma}(t)$ attains  the maximum  at a unique point $\tau_x^t$ of $[0,t]$ $\mu$--a.s. Moreover, by \eqref{e4.5c} we have 
\begin{equation}
\label{e16}
MS_{b,\sigma}(t)(x)=\sigma{\mathds 1}_{[0,\tau_x^t]}.
\end{equation}
Set $I=(a,+\infty)$ where $a>0$ is fixed.
By proceeding as in the proof of Proposition \ref{p4} we show
\begin{Proposition}
\label{p5}
The function
$$
g(x)=S_{b,\sigma}(1)(x),\quad x\in H,
$$
fulfills Hypothesis \ref{h1} on $I$  with
\begin{equation}
\label{e111}
u_t(x)=\psi(B_{b,\sigma}(t)x),\quad \gamma(x)=\sigma\int_0^1\psi(B_{b,\sigma}(t)x)\,{\rm d} t,
\end{equation}
where   $\psi:[0,+\infty)\to [0,+\infty)$ is  $C^\infty$ and such that \eqref{e11bis} holds.
\end{Proposition}\medskip

We finally consider the {\em Geometric Brownian motion}  $$
 X(t):=e^{(b-\frac12\,\sigma^2)t+\sigma \,B(t)},\quad\forall\;t\in[0,1],
 $$
 where $b\in\R$  and $\sigma>0$. $X(t),\,t\in[0,1],$ is the strong solution of the following SDE:
 $$
 {\rm d}X=bX \,  {\rm d}t+\sigma X \,  {\rm d}B(t),\quad X(0)=1.
 $$
 Set
$$
S_X(t)=\max_{s\in[0,t]}X(s),\quad\forall\;t\in[0,1].
$$
 Arguing as before we see  that  $\mu$--a.s.  $X(s),\,s\in[0,t],$ attains  the maximum  at a unique point $\tau_x^t$ of $[0,t]$ and that
 $$
 MS_X(t)(x)=\sigma\, S_X(t){\mathds 1}_{[0,\tau_x^t]},\quad\forall\;t\in[0,1].
 $$

Now let us fix $a>0$  and set $I=(a,+\infty)$.
By proceeding as in the proof of Proposition \ref{p4} we show
\begin{Proposition}
\label{p6}
The function
$$
g(x)=  S_X(1)(x),\quad x\in H,
$$
fulfills  Hypothesis \ref{h1}  with
\begin{equation}
\label{e20}
u_t(x)=\psi(S_X(t)(x)),\quad \gamma(x)=\sigma\,S_X(1)(x)\int_0^1\psi(S_X(t)(x))\,{\rm d} t,
\end{equation}
where   $\psi:[0,+\infty)\to [0,+\infty)$ is  $C^\infty$ and such that  \eqref{e11bis} is fulfilled.
\end{Proposition}
\begin{Remark}
\label{r3.7f}
\em
Let us consider  the {\em Brownian Bridge},
   $$
 B_0(t)=B(t)-tB(1),\quad t\in [0,1].
  $$
  Also  $B_0(t),\,t\in[0,1],$ attains  the maximum  at a unique point $\tau_x^t$ of $[0,t]$, $\mu$--a.s..   Set
\begin{equation}
\label{e23}
S_0(t)=\max_{s\in[0,t]}(X(t)),\quad\forall\;t\in[0,1],
\end{equation}
so that
\begin{equation}
\label{e24}
MS_0(t)(x)=B(\tau_x)-\tau_x,
\end{equation}
In this case we are not able to show that the function
$$
g(x)=\max_{t\in [0,1]}S_0(t)(x),
$$
fulfills Hypothesis \ref{h1}.
One realizes in fact   that,   due to the additional term $-\tau_x$ in equation \eqref{e24},  the proof of Proposition \ref{p4} does not work in this case.\medskip

A similar difficulty arises with  the  {\em Ornstein--Uhlenbeck} process
$$
 X(t)=\int_0^t e^{-(t-s)a}\,{\rm d} B(s),\quad t\in[0,1],
  $$
  where $a>0$ is fixed. 
  The  Malliavin derivative of $X(t)$ is given by
\begin{equation}
\label{e22a}
M_\tau X(t)x=e^{-a(t-\tau)}{\mathds 1}_{[0,t]}(\tau).
\end{equation}
 Setting
\begin{equation}
\label{e21}
S(1)=\max_{t\in[0,1]}(X(t)), 
\end{equation}
we have by \eqref{e4.5c}
\begin{equation}
\label{e22}
MS(1)(x)=e^{-a(1-\tau_x)}{\mathds 1}_{[0,\tau_x]},
\end{equation}
where $\tau_x$ is the the time when $X(t)(x)$ reaches the maximum value  as $t\in [0,1]$.

As before we  cannot  repeat the proof of Proposition \ref{p4} due to the exponential term in equation \eqref{e22}.
\end{Remark}

 \section{Integration by parts formulae in $L^2_+(0,1)$}
 
 \subsection{Setting of the problem}
 We are given   a   Gaussian process  $X(t),\,t\in[0,1],$ in $H=L^2(0,1)$. Its law is a Gaussian measure  which we denote by $\mu=N_Q$. We assume that $\mu$ is non degenerate. The following integration by parts formula  is well known
\begin{equation}
\label{e4.1f}
\int_H \langle D\varphi(x),z   \rangle\def\xx{_H}\,\mu({\rm d} x)=\int_H \varphi(x)\,\langle Q^{-1/2}x, Q^{-1/2}z  \rangle\def\xx{_H}\,\mu({\rm d} x),
\end{equation}
for all $\varphi\in C^1_b(H)$ and all $z\in Q^{1/2}(H)$.  
We shall also assume that $\mu$ is concentrated in $E:=C([0,1])$ and that the Cameron--Martin space $Q^{1/2}(H)$ is included in $E$.
Then we can find easily an integration by parts formula on $E$. For this we need the following elementary lemma, see e.g., \cite[Lemma 2.1]{CeDa14}.

\begin{Lemma}
\label{l4.1f}
For any $\varphi\in C^1_b(E)$ there exists a sequence $(\varphi_n)\in C^1_b(H)$ such that
\begin{itemize}
\item[(i)] $\ds\lim_{n\to \infty}\varphi_n(x)\to \varphi(x),\quad \forall\;x\in E$.
\item[(ii)] $\ds 
\lim_{n\to \infty}\langle D\varphi_n(x),z\rangle_H=D\varphi(x)\cdot z,$ for all $x,z\in E$.
\end{itemize}
\end{Lemma}

Now we can prove

\begin{Proposition}
\label{p4.2f}
For all $\varphi\in C^1_b(E)$ and any $z\in Q^{1/2}(H)$ the following integration by parts formula holds
\begin{equation}
\label{e4.2f}
\int_E D\varphi(x)\!\cdot \!z    \,\mu({\rm d} x)=\int_E \varphi(x)\,\langle Q^{-1/2}x, Q^{-1/2}z  \rangle\def\xx{_H}\,\mu({\rm d} x).
\end{equation}
\end{Proposition}

\begin{proof}
Let $\varphi_n\in C^1_b(H)$ be a sequence  
as in Lemma \ref{l4.1f}. Then by \eqref{e4.1f} we have
\begin{equation}
\label{e4.3f}
\int_H \langle D\varphi_n(x),z\rangle_H  \,\mu({\rm d} x)
= \int_H \varphi_n(x)\,\langle Q^{-1/2}x, Q^{-1/2}z  \rangle\def\xx{_H}\,\mu({\rm d} x).
\end{equation}
The conclusion follows letting $n\to\infty.$
\end{proof}

\begin{Corollary}
For all $\varphi,\psi\in C^1_b(E)$ and any $z\in Q^{1/2}(H)$ the following integration by parts formula holds
\begin{equation}
\label{e4.2g}
\int_E D\varphi\!\cdot \!z \,\psi \, {\rm d} \mu
= - \int_E D\psi\!\cdot \!z \,\varphi \, {\rm d} \mu +
\int_E \varphi\,\psi\,\langle Q^{-1/2}x, Q^{-1/2}z  \rangle \, {\rm d} \mu.
\end{equation}
\end{Corollary}

\begin{Remark}
\label{r4.3f}
\em
By \eqref{e4.2f} it follows, by standard arguments, that the gradient operator $D$ is closable in $L^p(E,\mu)$ for any $p\ge 1$, we shall still denote by $D$ its closure. As a consequence  identity  \eqref{e4.2f} also holds for $\varphi$ belonging to the domain of the closure of $D$.
\end{Remark}

The main object of this section is the following function in $E$
\begin{equation}
\label{e4.4f}
g(x):=\min_{t\in [0,1]} X(t)(x),\quad x\in E.
\end{equation}
We shall assume that
\begin{Hypothesis}
\label{h2}
$\null$\par\noindent
(i)  The law of $g$ is absolutely continuous with respect to  the Lebesgue measure $\lambda$ on $(-\infty,0]$; we shall denote  $\rho$  the corresponding density.

\medskip

\noindent(ii) For  $\mu$--almost all  $x\in E$,  trajectories of $X(t)$ attain their minimum
$g(x)$   at   a unique point   $\tau_x\in [0,1]$. Moreover,   there exists the directional derivative of $g(x)$ in all directions $z\in E$ and its result
  \begin{equation}
\label{e4.4l}
Dg(x)\!\cdot \!z:=z(\tau_x).
\end{equation}
\end{Hypothesis}
 \begin{Proposition}
\label{p6.1}
Assume  Hypothesis \ref{h2}.
Then for all $\varphi\in C^1_b(E)$ and all $z\in Q^{1/2}(H)$ the following identity holds 
 for every $r < 0$:
\begin{multline}
\label{e6.2}
\E\big[z(\tau)\varphi\big | g\!=\!r\big] \, \rho(r)
-\int_{\{g\ge r\}} \left( D \varphi\! \cdot \!z - \langle Q^{-1/2}x,Q^{-1/2}z\rangle \, \varphi \right) \, {\rm d} \mu.
\end{multline}
\end{Proposition}

\begin{proof}
We fix  $r<0$ and $\epsilon>0$ such that $ r+\epsilon<0$. Then  we apply \eqref{e4.2g} setting $\psi=\theta_\epsilon(g) $, where $\theta_\epsilon$ is given by 
\begin{align*}
\theta_\eps(\xi)=
\begin{cases} 
0\quad & \text{ if } \quad \xi<r-\epsilon
\\
\frac{\xi-r+\epsilon}{2\epsilon}
\quad & \text{ if } \quad \xi\in [r-\epsilon,r+\epsilon]
\\
1\quad & \text{ if } \quad \xi> r+\epsilon.
\end{cases}
\end{align*}
By the chain rule we have
\begin{align*}
D(\theta_\epsilon(g))=\theta'_\epsilon(g)Dg=\frac1{2\epsilon}\,{\mathds 1}_{\{r-\epsilon\le g\le r+\epsilon\}}\,Dg,
\end{align*}
and by \eqref{e4.4l} we deduce
\begin{multline}
\label{e6.3}
\frac1{2\eps}\int_E \mathds 1_{\{r-\epsilon\le g\le r+\epsilon\}}  z(\tau)  \,\varphi \, {\rm d} \mu
\\
= - \int_{\{g\ge r-\epsilon\}} \theta_\epsilon(g) \left(  D \varphi \cdot z - \langle Q^{-1/2}x,Q^{-1/2}z\rangle \, \varphi \right) \, {\rm d} \mu.
\end{multline}
On the one hand,
\begin{align*}
\frac1{2\epsilon}\int_E {\mathds 1}_{\{r-\epsilon\le g\le r+\epsilon\}}z(\tau)\, \varphi\, {\rm d} \mu
=\frac1{2\epsilon}\int_E {\mathds 1}_{\{r-\epsilon\le g\le r+\epsilon\}}\E[z(\tau)\varphi|g]\, {\rm d} \mu.
\end{align*}
Since the law of $g$ is  given by the measure $\rho(r) \,  {\rm d} r$ we have
\begin{equation}
\label{e6.4}
 \frac1{2\eps}\int_{\{r-\epsilon\le g\le r+\epsilon\}} z(\tau)\,\varphi \, {\rm d} \mu
= \frac1{2\eps}\int_{r-\epsilon}^{r+\epsilon} \E\big[z(\tau)\,\varphi\big | g\!=\! \xi \big] \, \rho(\xi) \, {\rm d} \xi.
\end{equation}
Now   letting $\epsilon\to 0$,
we have that for almost every $r < 0$
\begin{align*}
\lim_{\eps \to 0} \frac1{2\eps}\int_{\{r-\epsilon\le g\le r+\epsilon\}} z(\tau)\,\varphi \, {\rm d} \mu
= \E\big[z(\tau)\,\varphi\big | g\!=\! r \big] \, \rho(r).
\end{align*}
We next consider the right hand side of \eqref{e6.3}, and we prove that it
converges to  \footnote{if $\psi(x) = D\varphi(x) \!\cdot \!z - \varphi(x) \langle Q^{-1/2}x,Q^{-1/2}z \rangle$ then $\E|\psi|^2 \le \|\varphi\|^2_{C^1_b} \left( |z|^2 + |Q^{-1/2}z|^2\right)$, and since the right hand side of \eqref{e6.3} is
dominated by $\psi$ and $\theta_\eps(g(x)) \to \mathds 1_{[r,+\infty)}(g(x))$, we can apply the dominated convergence theorem}
\begin{align*}
-\int_{\{g\ge r\}} \left( D \varphi \cdot z - \langle Q^{-1/2}x,Q^{-1/2}z\rangle \, \varphi \right)\,d\mu
\end{align*}
and since this is a continuous function in $r$  \footnote{with the notation of formula \eqref{e1.1},
we shall prove that $F_\psi(r)$ is a continuous function; notice that $F_\psi(r+\eps) - F_\psi(r) = \int \theta_\eps'(g(x)) \psi(x) \, \mu(dx)$ and, again, we can apply the dominated convergence theorem by noticing that
$\P(g(x) \in [r,r+\eps)) = \int_r^{r+\eps} \rho(t) \, {\rm d} t \to 0$ due to the absolute continuity of the law of $g$
}, formula \eqref{e6.2} follows for every $r < 0$.
\end{proof}

 \begin{Remark}
 \label{r6.2}\em
 Assume, besides the assumptions of Proposition \ref{p6.1} that $\varphi\langle Mg,z\rangle\in \mathbb D^{1,2}(H,\mu)$ or $\varphi\langle Mg,z\rangle\in UC_b(H)$.   Then if $r<0$, identity \eqref{e6.2} can be written as
 \begin{equation}
\label{e6.2a}
  \int_{\{g= r\}}\langle Mg,z\rangle\,\varphi\,  {\rm d}\sigma_r =-\int_{\{g\ge r\}} \langle M\varphi,z\rangle\,  {\rm d}\mu+\int_{\{g\ge r\}}   W_z\, \varphi\,  {\rm d}\mu,
\end{equation}
where $\sigma_r$ is the surface measure  introduced before.
\end{Remark}

\medskip
The following result will be useful later to pass to the limit for $r\to 0$.
\begin{Proposition}\label{p4.7T}
Assume, besides Hypothesis \ref{h2}, that the joint law of the random vector $(g,\tau)$ can be written as
$\pi({\rm d} y, {\rm d} s) = \pi(y,s) \, {\rm d} y {\rm d} s$ on $\bR \times [0,1]$, with the map $y\to \pi(y,s)$ continuous.
Then for all $\varphi\in C^1_b(E)$ and all $z\in Q^{1/2}(H)$ the following identity holds $\forall\;r<0$
\begin{multline}
\label{e6.star}
\int_0^1 \E\big[\varphi| g\!=\!r, \tau\!=\!s\big] \, z(s) \,\pi(r,s)\,{\rm d} s
\\
= - \int_{\{g\ge r\}}[D\varphi\cdot z- \langle Q^{-1/2}x, Q^{-1/2}z  \rangle\, \varphi]\,{\rm d} \mu,\quad \forall\;r<0.
\end{multline}
\end{Proposition}
 
\begin{proof}
Let us consider the left hand side of \eqref{e6.3}: by 
conditioning with respect to $\{g=y,\tau=s\}$ we get
\[
\frac1{2\eps}\int_{\{r-\epsilon\le g\le r+\epsilon\}}  z(\tau) \,\varphi\,  {\rm d}\mu =
\frac{1}{2\eps} \int_{r-\eps}^{r+\eps} \int_0^1 z(s) \,  \E\big[\varphi| g\!=\!y, \tau\!=\!s\big] \, \pi({\rm d} y, {\rm d} s)
\]
Passing to the limit for $\eps \to 0$ and taking into account 
that $\pi({\rm d} y, {\rm d} s) =\pi(y,s) \, {\rm d} y \, {\rm d} s$, we get
the thesis.
\end{proof}

\subsection{Examples}

\subsubsection{Brownian motion}\label{s4.2.1}

Let $X(t)=B(t),\,t\in[0,1]$. Then the law $\mu$ of $X$ is  the Wiener measure, it is  is concentrated on   $\{x\in C([0,1]):\,x(0)=0,\;x'(1)=0\} \subset E$.
Moreover,  the law of $g$ (defined by \eqref{e4.4f})
is given by,
see e.g. \cite[1.2.4  page 154]{BoSa02},  \begin{equation}
\label{e6.7d}
(\mu\circ g^{-1})({\rm d} r)=\frac2{\sqrt {2\pi}}e^{-\frac12r ^2}\,{\mathds 1}_{(-\infty,0]} \, {\rm d} r
\end{equation}
 As well known, for almost all  $x\in E$  $B(\cdot) x$ has a unique minimum point at $\tau_x$ so that Hypothesis \ref{h2} is fulfilled  see e.g. \cite{EnSt93} .

Consequently, applying the   integration by parts formula   \eqref{e6.2}, we obtain   that for all   $\varphi\in C^1_b(H)$, $z\in Q^{1/2}(H)$ and all $r<0$ we have
\begin{multline}
\label{e6.9d}
\frac2{\sqrt {2\pi}}e^{-\frac12r^2} \E\left[ z(\tau)\,\varphi|g\!=\!r\right]  
\\
\ds=-
\int_{\{g\ge r\}} (D\varphi(x)\!\cdot \!z-  \varphi(x) \,\langle Q^{-1/2}x,Q^{-1/2}z\rangle)\, \mu({\rm d} x).
\end{multline}

We want now to   extend identity  \eqref{e6.9d} up to $r=0$. To this purpose we    recall two facts. The first one is a result  from
\cite[page 118, Thm 2.1]{DuIgMi77}; namely defining the probability measure $\nu_r$, $r<0,$ as
\begin{equation}
\label{e4.16f}
\int_H \varphi\, {\rm d} \nu_r = \tfrac1{\mu(\{g\ge r\})}\int_{\{g\ge r\}}\varphi\, {\rm d} \mu,\quad \varphi\in C_b(H)
\end{equation}
then $\nu_r$ converges weakly to $\nu$, where $\nu$ is the law of the Brownian meander. 

The second one is the expression for the joint density  $\pi(y,s)$ of $g=B(\tau)$ and $\tau$, which is given by, see \cite[1.14.4 page 172]{BoSa02}
\begin{equation}
\label{e4.17f}
\pi(y,s)=\tfrac{|y|}{\sqrt{\pi^2\,s^3\,(1-s)}}\;e^{-\frac{y^2}{2s}},\quad y\le 0,\,s\in[0,1].
\end{equation}
Now we can prove

\begin{Proposition}
\label{p4.9h}
For all $\varphi\in C^1_b(E)$ and all $z\in Q^{-1/2}(H)$ we have
  \begin{multline}
\label{e4.20f}
\tfrac1{\sqrt{2\pi}} \int_0^1\E\left[\varphi|g\!=\!0,\tau\!=\!s\right] 
  \tfrac{z(s)}{\sqrt{\,s^3\,(1-s)}} \, {\rm d} s
\\
=-
\int_{\{g=0\}} [D \varphi(x)\!\cdot \!z -  \varphi (x)\langle Q^{-1/2}x,Q^{-1/2}z\rangle]\, \nu({\rm d} x),
\end{multline}
where $\nu$ is the law of the Brownian meander.
\end{Proposition}

\begin{proof}
By \eqref{e6.star}  we have, taking into account
\eqref{e4.17f},
\begin{multline}
\label{e4.18f}
- \int_{\{g\ge r\}}[D\varphi\!\cdot\! z- \langle Q^{-1/2}x, Q^{-1/2}z  \rangle\, \varphi]\, {\rm d} \mu
 \\
 =
  \int_0^1\E\left[\varphi|g\!=\!r,\tau\!=\!s\right] 
  \tfrac{|r|\,z(s)}{\sqrt{\pi^2\,s^3\,(1-s)}}\;e^{-\frac{r^2}{2s}}\, {\rm d} s.
  \end{multline}
 On the other hand, by \eqref{e6.7d} we have
\begin{align*}
 \mu(g\ge r)= \tfrac2{\sqrt {2\pi}}\int^0_{r}e^{-\frac12s ^2} \, {\rm d} s
\end{align*}
it follows that
\begin{align*}
\lim_{r\to 0}\,\tfrac{ \mu(g\ge r)}{|r|}=\frac2{\sqrt{2\pi}}.
\end{align*}
Therefore, dividing both sides of \eqref{e4.18f} by $\mu(g\ge r)$ and letting $r$ tend to zero 
we find identity  \eqref{e4.20f}. 
 \end{proof}

 \begin{Remark}
 \label{r1z}
 \em Identity \eqref{e4.20f} was proved by a different method by S. Bonaccorsi and L. Zambotti, \cite{BoZa04}.
 \end{Remark}
 
\subsubsection{Distorted Brownian motion}
Now we consider a {\em distorted Brownian motion}
\begin{equation}
\label{e4.15}
X(t):= b\,t+\sigma B(t),\quad\forall\;t\in[0,1],
\end{equation}
where $b>0$ and $\sigma>0$ are given. Set as usual
\begin{equation}
\label{e4.15f}
g:=\inf_{s\in[0,1]}(b \, s+\sigma B(s)).
\end{equation}

Then the law $\mu$ of $X$ is  the Wiener measure.
Moreover,  the law of $g$ (defined by \eqref{e4.15f}) has a density with respect to the Lebesgue measure given by,
see e.g. \cite[1.2.4  page 251]{BoSa02} \footnote{
\begin{align*}
\mathbb P_0(\inf_{0 \le t \le 1} \mu s + \sigma W_s \le y) = \mathbb P_0(\inf_{0 \le t \le \sigma^2} \frac\mu{\sigma^2} t + W_t \le y) 
\\
= \frac12 \mathop{\rm Erfc}\left(-\frac{y-\mu}{\sigma \sqrt{2}} \right) + \frac12 e^{2 \mu y/\sigma^2}
\mathop{\rm Erfc}\left(-\frac{y+\mu}{\sigma \sqrt{2}} \right) 
\end{align*}}
\begin{equation}
\label{e4.20T}
\rho_{b,\sigma}(r) = \tfrac{\sqrt{2}}{\sigma \sqrt{\pi}} \exp\left(-\tfrac{(r-b)^2}{2 \sigma^2}\right) + \tfrac{b}{\sigma^2} e^{2 \frac{b}{\sigma^2} r} \mathop{\rm Erfc}\left(- \tfrac{r+b}{\sigma \sqrt{2}} \right)
%
%
\,{\mathds 1}_{(-\infty,0]}(r).
\end{equation}
By a direct computation, we have
\begin{align*}
\mu(g \ge r) \approx C_{b,\sigma} r + o(r), \qquad r \to 0,
\end{align*}
where $C_{b,\sigma}$ can be explicitly computed.\footnote{
$\ds C_{b,\sigma} = \frac{b}{\sigma^2} \left( \mathop{\rm Erfc}\left(\tfrac{b}{\sigma \sqrt{2}}\right) - 2 \right) - \tfrac{\sqrt{2}}{\sqrt{\pi \sigma^2}} e^{-b^2/2\sigma^2}$}
\\
For almost all  $x\in E$,  $X(\cdot) x$ has a unique minimum point at $\tau_x$ 
so that Hypothesis \ref{h2} is fulfilled (see e.g. Proposition \ref{p3.1}).

Consequently, applying the   integration by parts formula   \eqref{e6.2}, we obtain   that for all   $\varphi\in C^1_b(H)$, $z\in Q^{1/2}(H)$ and all $r<0$ we have
\begin{multline}
 \label{e4.21T} 
\rho_{b,\sigma}(r) \, \E\left[ z(\tau)\,\varphi|g=r\right]  
\\
=-
\int_{\{g\ge r\}} (D\varphi(x)\cdot z-  \varphi(x) \,\langle Q^{-1/2}x,Q^{-1/2}z\rangle)\, \mu({\rm d} x).
\end{multline}
Now, to apply Proposition \ref{p4.7T} we need the expression of the joint density of $g$ and $\tau$ which is given by,
see e.g. \cite[1.13.4  page 268]{BoSa02}
\begin{equation}
\label{e4.22z}
\pi(r,s)=\tfrac{|r|}{\sqrt{\pi\sigma^2}\,s^{3/2}}e^{-\frac{(|r|+b\, s)^2}{2\sigma^2s}}
\left(\tfrac{e^{- \frac{b^2}{2\sigma^2}(1-s)}}{\sqrt{\pi\sigma^2(1-s)}}+\tfrac{b}{\sqrt{2}\sigma^2}\mbox{Erfc}\,\left(-\tfrac{b\,\sqrt{1-s}}{\sqrt{2\sigma^2}}\right)\right)
\end{equation}

Now we can prove
\begin{Proposition}
\label{p4.9hh}
For all $\varphi\in C^1_b(E)$ and all $z\in Q^{1/2}(H)$ we have
  \begin{multline}
\label{e4.20fff}
  \int_0^1\E\left[\varphi|g\!=\!0,\tau\!=\!s\right] \tilde\pi(s)\, {\rm d} s
\\
=-\tfrac{1}{\int_{\{g=0\}}e^{-b\,x(1)}\, \nu({\rm d} x)}
\int_{\{g=0\}}e^{-b\,x(1)} [D \varphi(x)\cdot z -  \varphi (x)\langle Q^{-1/2}x,Q^{-1/2}z\rangle]\, \nu({\rm d} x),
\end{multline}
where  $\tilde\pi$ is a given function, that is defined in \eqref{eq:2707-1} below, and $\nu$ is the law of the Brownian meander.
\end{Proposition}

\begin{proof}
 First we notice that 
for all $\varphi\in C^1_b(E)$ and all $z\in Q^{-1/2}(H)$ we have
\begin{multline}
\label{e4.23jj1}
 \int_0^1\E\left[ \varphi | g\!=\!r,\tau\!=\!s\right]
z(s)\,\pi(r,s)\,{\rm d} s\\
\\
=
-\int_{\{g\ge r\}}  [ D \varphi(x)\!\cdot \!z-  \varphi(x)\,\langle Q^{-1/2}x,Q^{-1/2}z\rangle ]\, \mu({\rm d} x).
\end{multline}
Now we have to divide both sides of \eqref{e4.23jj1} by $\mu(g\ge r)$
and pass to the limit for $r\uparrow 0$. 
\\
By using the explicit formulas available for $\pi(r,s)$ and $\mu(g \ge r)$, which show that these are (asymptotically) linear functions for $r$ around 0, we have
\begin{align}
\nonumber
\lim_{r \uparrow 0} \frac{\pi(r,s)}{\mu(g \ge r)} &= \tfrac{1}{C_{b,\sigma} \, \sqrt{\pi\sigma^2}\,s^{3/2}} 
\left(\tfrac{e^{- \frac{b^2}{2\sigma^2}}}{\sqrt{\pi\sigma^2(1-s)}}+\tfrac{b}{\sqrt{2}\sigma^2} e^{-\frac{b^2}{2\sigma^2}s} \, \mbox{Erfc}\,\left(-\tfrac{b\,\sqrt{1-s}}{\sqrt{2\sigma^2}}\right)\right)
\\
\label{eq:2707-1}
&= \tilde \pi(s).
\end{align}
or the right hand side we will use Girsanov's theorem. 
Let us consider the
process $Y(t) = \frac{1}{\sigma}X(t) = \frac{b}{\sigma}t + B(t)$, and the probability measure
\begin{align*}
\gamma({\rm d} x)=\rho(x) \mu({\rm d} x) = e^{\frac{b}{\sigma} B(1) - \frac12 \frac{b^2}{\sigma^2}} \, \mu({\rm d} x).
\end{align*}
Then $Y$ is a Brownian motion on  $(H,\mathcal B(H),\gamma)$.
On the other hand
\begin{align*}
\mu({\rm d} x)=\rho^{-1}(x)\,\gamma({\rm d} x) = e^{-\frac{b}{\sigma} Y(1) - \frac12 \frac{b^2}{\sigma^2}} \, \gamma({\rm d} x).
\end{align*}
We let $\tilde g = \min_{s \in [0,1]} Y(s) = \frac{1}{\sigma}g$; 
since $Y$ is a Brownian motion under $\gamma$, we can reason as in previous section to prove that there exists the limit
\begin{align*}
\lim_{r\to 0}\tfrac1{\gamma(\tilde g\ge r/\sigma)}\int_{\{\tilde g\ge r/\sigma\}}\varphi \, {\rm d} \gamma = 
\int_{\{g=0\}}\varphi \, {\rm d} \nu,\quad \forall\;\varphi\in C_b(H),
\end{align*}
where $\nu$, as in subsection \ref{s4.2.1}, is the law of  the Brownian meander.
\\
Then it follows, by taking the limit after a change of measure, that
\begin{align*}
\tfrac{1}{\mu(g\ge r)}\int_{\{g\ge r\}}\varphi \, {\rm d} \mu
=\tfrac{\int_{\{\tilde g\ge r/\sigma\}}\varphi\,\rho^{-1}\, d\gamma}{\int_{\{\tilde g\ge r/\sigma\}}\rho^{-1}\,d\gamma}
&=
\tfrac{
\tfrac
{1}
{\gamma(\tilde g\ge r/\sigma)}\int_{\{\tilde g\ge r/\sigma\}}\varphi\,\rho^{-1}\, {\rm d} \gamma
}
{
\frac
{1}
{\gamma(\tilde g\ge r/\sigma)}\int_{\{\tilde g\ge r/\sigma\}}\rho^{-1}\, {\rm d} \gamma
}
\\
&\longrightarrow
\tfrac{\int_{\{g=0\}}\varphi\,\rho^{-1}\, {\rm d} \nu}{\int_{\{g=0\}}\rho^{-1}\, {\rm d} \nu}
\end{align*}
that we can write in the form
\begin{align*}
\tfrac{1}{\mu(g\ge r)}\int_{\{g\ge r\}}\varphi \, {\rm d} \mu
\longrightarrow 
\tfrac{\int_{\{g=0\}} \varphi\,e^{-b\,x(1)}\, \nu({\rm d} x)}{\int_{\{g=0\}}e^{-b\,x(1)}\, \nu({\rm d} x)}.
\end{align*}
So, the conclusion follows.
\end{proof}

\def\temp{We want now to  to extend identity  \eqref{e6.9d} up to $r=0$. To this purpose we    recall two facts. The first one is a result  from
\cite[page 118, Thm 2.1]{DuIgMi77}; namely defining the probability measure $\mu_r,\, r<0,$ as
\begin{equation}
\label{e4.16f}
\int_H \varphi\,d\nu_r=\frac1{\mu(\{g\ge r\})}\int_{\{g\ge r\}}\varphi\,d\mu,\quad \varphi\in C_b(H)
\end{equation}
then $\nu_r$ converges weakly to $\nu$, where $\nu$ is the law of the Brownian meander. 

The second one is the expression for the joint density  $\psi(y,s)$ of $g=B(\tau)$ and $\tau$, which is given by, see \cite[1.14.4 page 172]{BoSa02}
\begin{equation}
\label{e4.17f}
\psi(y,s)=\frac{|y|}{\sqrt{\pi^2\,s^3\,(1-s)}}\;e^{-\frac{y^2}{2s}},\quad y\le 0,\,s\in[0,1].
\end{equation}
Now we can prove
\begin{Proposition}
\label{p4.9h}
For all $\varphi\in C^1_b(E)$ and all $z\in Q^{-1/2}(H)$ we have
  \begin{equation}
\label{e4.20f}
\begin{array}{l}
\ds\frac1{\sqrt{2\pi}} \int_0^1\E\left[\varphi|g=0,\tau=s\right] 
  \frac{z(s)}{\sqrt{\,s^3\,(1-s)}} 
  \,ds\\
  \\
  \ds=-
\int_{\{g=0\}} [D \varphi(x)\cdot z -  \varphi (x)\langle Q^{-1/2}x,Q^{-1/2}z\rangle]\, \nu(dx),
\end{array}
  \end{equation}
where $\nu$ is the law of the Brownian Meander.
\end{Proposition}
\begin{proof}
By \eqref{e6.9d} we have, taking into account
\eqref{e4.17f},
\begin{equation}
\label{e4.18f}
\begin{array}{l}
 \ds\frac2{\sqrt {2\pi}}e^{-\frac12r^2} \E\left[ z(\tau)\,\varphi|g=r\right]\\
 \\
 \ds =
  \int_0^1\E\left[\varphi|g=r,\tau=s\right] 
  \frac{|r|\,z(s)}{\sqrt{\pi^2\,s^3\,(1-s)}}\;e^{-\frac{r^2}{2s}}\,ds.
  \end{array}
\end{equation}
 On the other hand, by \eqref{e6.7d} we have
 $$
 \mu(g\ge r)= \frac2{\sqrt {2\pi}}\int^0_{r}e^{-\frac12s ^2}ds
 $$
 it follows that
 $$
\lim_{r\to 0}\,\frac{ \mu(g\ge r)}{|r|}=\frac2{\sqrt{2\pi}}.
 $$
 Therefore, dividing both sides of \eqref{e4.18f} by $\mu(g\ge r)$ and letting $r$ tend to zero, yields
 \begin{equation}
\label{e4.19f}
\begin{array}{l}
 \ds\lim_{r\to 0}\,\frac1{ \mu(g\ge r)}\,\frac2{\sqrt {2\pi}}e^{-\frac12r^2} \E\left[ z(\tau)\,\varphi|g=r\right]\\
 \\
 \ds =
  \sqrt{\frac{\pi}{2}}   \int_0^1\E\left[\varphi|g=0,\tau_x=s\right] 
  \frac{z(s)}{\sqrt{\pi^2\,s^3\,(1-s)}} 
  \,ds
  \end{array}
\end{equation}
 Finally, letting $r\to 0$ in both sides of \eqref{e6.9d}, we find identity  \eqref{e4.20f}. \end{proof}

\begin{Remark}
\label{r1z}
\em Identity \eqref{e4.20f} was proved by a different method by S. Bonaccorsi and L. Zambotti, \cite{BoZa04}.
\end{Remark}}

\subsubsection{ Brownian bridge}
\label{ex6.6}

Let us  consider now the Brownian bridge, 
\begin{align*}
X(t)=B_0(t)=B(t)-tB(1),\quad t\in[0,1],
\end{align*}
whose law $\mu$   is concentrated on  the Banach space $\{x\in C([0,1]:\,x(0)=x(1)=0\}\subset E$. The law of $g$ can be obtained easily by conditioning  with respect to $B(1)$, see \cite[1.2.8  page 154]{BoSa02} 
\begin{equation}
\label{e6.7e}
(\mu\circ g^{-1})({\rm d} r)=4\, |r|\,e^{-2r^2}\,{\rm d} r\,{\mathds 1}_{(-\infty,0]},\qquad r<0.                                          
\end{equation}
Moreover, also in this case Hypothesis \ref{h2} is fulfilled.
 Therefore, 
by formula \eqref{e6.2}   we deduce that for any  $z\in H$ and any $r<0$
\begin{multline}
 \label{e6.9e}
4\,|r|\,e^{-2r^2} \;\E\left[ z(\tau)\,\varphi|g=r\right] \\
\\
=
-\int_{\{g\ge r\}} [ D \varphi(x)\!\cdot\! z-  \varphi(x)\,\langle Q^{-1/2}x,Q^{-1/2}z\rangle ]\, \mu(dx).
\end{multline}
Now for letting $r\to 0$ we proceed as before.   First we notice that  defining the probability measure $\mu_r,\, r<0,$ as
\begin{equation}
\label{e4.16ff}
\int_H \varphi\,{\rm d} \nu_r=\tfrac1{\mu(g\ge r)}\int_{\{g\ge r\}}\varphi\,{\rm d} \mu,\quad \varphi\in C_b(H)
\end{equation}
then by \cite[page126, Thm 5.1]{DuIgMi77} $\nu_r$ converges weakly to $\nu$, where $\nu$ is the law of the $3$--D Bessel bridge.

Moreover, we can deduce   the joint density  $\psi_0(y,s)$ of $g=B_0(\tau)$ and $\tau$ by   conditioning  a previous formula with respect to $B(1)$, that is
\[
\mu\left(\tau\in ds,\inf_{0\le t\le 1}B_0(t)\in dr\right)=\mu\left(\tau\in ds,\inf_{0\le t\le 1}B(t)\in dr\mid B(1)=0\right),
\]
using the following expression in \cite[1.14.8 page 173]{BoSa02}
\begin{multline*}
\mu(\tau\in ds,\inf_{0\le t\le 1}B(t)\in dr,B(1)\in dz)
\\
=
\tfrac{|r|(z+|r|)}{\pi\sqrt{s^3(1-s)^3}} \exp\left(-\tfrac{r^2}{2 s}-\tfrac{(z+|r|)^2}{2(1-s)}\right)\;\ds \,{\rm d} r\,{\rm d} z ,
\end{multline*}
with $r< 0\wedge z$. Finally, we obtain
\begin{multline}
\label{e4.23j}
\psi_0(y,s)dy\,ds=\mu(\tau\in ds,\inf_{0\le t\le 1}B_0(t)\in dy)\\
\\
= \sqrt{\tfrac{2}{\pi}}\;\tfrac{y^2}{\sqrt{s^3(1-s)^3}}e^{-\frac{y^2}{2s(1-s)}}\,{\rm d} s\,{\rm d} y.
\end{multline}
Now we can prove
\begin{Proposition}
\label{p4.9i}
For all $\varphi\in C^1_b(E)$ and all $z\in Q^{-1/2}(H)$ we have
\begin{multline}
\label{e4.23jj}
\tfrac{1}{\sqrt{2\pi}}\;\int_0^1\E\left[ \varphi | g\!=\!r,\tau\!=\!s\right]
\tfrac{z (s)}{\sqrt{s^3(1-s)^3}}\,ds\\
\\
=
-\int_{\{g=0\}}  [ D \varphi(x)\!\cdot \!z-  \varphi(x)\,\langle Q^{-1/2}x,Q^{-1/2}z\rangle ]\, \nu(dx).
\end{multline}
\end{Proposition}

\begin{proof}
By \eqref{e6.star}, taking into account \eqref{e4.23j}, we have
\[
 \begin{array}{l}
\ds \int_0^1\E\left[ \varphi | g\!=\!r,\tau=s\right]
 \sqrt{\tfrac{2}{\pi}}\;\tfrac{r^2\;(z(s))}{\sqrt{s^3(1-s)^3}}e^{-\tfrac{r^2}{2s(1-s)}}\,{\rm d} s\\
\\
\hspace{3cm}\ds=
-\int_{\{g\ge r\}}  [ D \varphi(x)\!\cdot \!z-  \varphi(x)\,\langle Q^{-1/2}x,Q^{-1/2}z\rangle ]\, \mu(dx).
\end{array}
\]
Since
\[
\mu(g\ge r)=4\int^0_{r} |s|\,e^{-2s^2}\,ds,
\]
 we have
$$
\lim_{r\to 0}\,\frac{ \mu(g\ge r)}{r^2}=2,
$$
and the conclusion follows.
\end{proof}

 \begin{Remark}
 \label{r2z}
 \em Identity \eqref{e4.23jj} was proved by a different method by   L. Zambotti, \cite{Za01}.
 We also quote \cite{Ot09} for a similar formula but assuming $r<0$.

 \end{Remark}

 \subsubsection{Ornstein--Uhlenbeck}
\label{ex6.8}

Let
$$
X(t)=\int_0^te^{-(t-s)a}\,{\rm d} B(s),\quad t\in[0,1],
$$
where $a>0$.
Then  the law of $g$ is given by, see \cite[1.2.4  Pag. 522]{BoSa02}
\begin{equation}
\label{e6.7f}
(\mu\circ g^{-1})(dr)=\frac{2}{\sqrt{\pi}}\frac{\sqrt{a}}{\sqrt{e^{2a}-1}}\; e^{-\frac{a\,r^2}{e^{2a}-1}}\;dr\;{\mathds 1}_{(-\infty,0]}(r),\qquad \forall r<0
\end{equation}
  
By formula \eqref{e6.2}   we deduce that for any  $z\in H$ and any $r<0$
\begin{equation}
 \label{e6.9f}
 \begin{array}{l}
\ds \frac{2}{\sqrt{\pi}}\frac{\sqrt{a}}{\sqrt{e^{2a}-1}}\; e^{-\frac{a\,r^2}{e^{2a}-1}}\;\E\left[
z(\tau) \, \varphi \mid g=r\right] \\
\\
\ds=-
\int_{\{g\ge r\}} [ D \varphi(x)\!\cdot \!z-  \varphi(x)\,\langle Q^{-1/2}x,Q^{-1/2}z\rangle ]\, \mu(dx).
\end{array}
\end{equation}
We emphasize once more that this formula holds for $r < 0$. As opposite to previous sections,
however,  at this point we do not know the existence of the joint density $\pi(r,s)$ of $g$ and $\tau$, therefore we are not able to
apply Proposition  \ref{p4.7T} and pass to the limit as $r \to 0$.

 \subsection{Some remarks about   the Neumann problem}
Let $X(t),\,t\in[0,1],$ be  the Gaussian process considered in Section 4.1 with law $\mu=N_Q$ assuming that Hypothesis \ref{h2} is fulfilled. We denote by $(e_h)$ and $(\lambda_h)$ eigen--sequences such that
\begin{align*}
Qe_h=\lambda_h e_h,\quad h\in\N.
\end{align*}
Let us consider the following stochastic  {equation in the infinite dimensional Hilbert space $H$}
\begin{equation}
\label{e1a}
\left\{\begin{array}{l}
{\rm d}Z=-\frac12\,Q^{-1}Z \, {\rm d}t+ {\rm d}W(t),\\
\\
Z(0)=x,
\end{array}\right. 
\end{equation}
where $W$ is a cylindrical white noise in $H$. It is well known that equation \eqref{e1a} has a unique mild  solution  $Z(t,x)$ (which is an Ornstein--Uhlenbeck process) and that $\mu=N_Q$ is invariant for the corresponding transition semigroup
\begin{equation}
\label{e4b}
P_t\varphi(x)=\E[\varphi(Z(t,x))],\quad \varphi\in C_b(H).
\end{equation} 
We denote by $\mathcal L$ the Kolmogorov operator
\begin{equation}
\label{e1}
\mathcal L\varphi=\frac12\,\mbox{\rm Tr}\,[D^2\varphi]-\frac12\,\langle    Q^{-1}x,D\varphi\rangle,\quad \varphi\in\mathcal E_A(H),
\end{equation} 
where $\mathcal E_A(H)$ is  the space of all exponential functions, see e.g. \cite{Da04}. We recall the identity
\begin{equation}
\label{e6b}
\int_H \mathcal L\varphi\; \psi\, {\rm d}\mu=-\frac12\int_H \langle D\varphi,D\psi   \rangle\, {\rm d} \mu,\quad \forall\;\varphi,\psi\in \mathcal E_A(H).
\end{equation}
Let $g$ be   the function defined by \eqref{e4.4f}  and set
\begin{align*}
K_r=\{x\in H: \, g(x) \ge r\}.
\end{align*}

We are interested in the Neumann problem in $K_r$.  For this we fix $r< 0$ and  introduce the symmetric Dirichlet form 
\begin{equation}
\label{e7b}
a_r(\varphi,\psi)=\frac12\int_{K_r} \langle D\varphi,D\psi   \rangle\,{\rm d} \mu.
\end{equation}
If  $a$ is closable (this had to be checked in the different situations), then thanks to the Lax--Milgram Lemma,  for any $\lambda>0$ and $f\in L^2(H,\mu)$ there exists   $\varphi\in W^{1,2}(H,\mu)$ such that
\begin{equation}
\label{e9b}
\lambda\int_{K_r}\varphi\,\psi\, {\rm d}\mu+a(\varphi,\psi)=\int_{K_r}f\,\psi\,{\rm d} \mu,\quad \forall\,\psi\in \mathcal E_A(H).
\end{equation}
  $\varphi$ is called the {\em weak} solution to the {\em Neumann} problem
\begin{equation}
\label{e10b}
\lambda\ \varphi-\mathcal L\varphi= f\quad\mbox{\rm in }\;K_r.
\end{equation}

 To see whether $\varphi$ fulfills a  Neumann type condition on the boundary of $K_r$ it is important to extend the integration formula \eqref{e6b} to $K_r$. This is provided by the following result.
  \begin{Proposition}
\label{p1}
Under the assumptions of Proposition 4.7, let  $\varphi,\,\psi\in {\mathcal E_A(H)}$, $r<0$  and assume that the series
\begin{equation}
\label{e4.41}
\langle Dg,D\varphi   \rangle:=\sum_{h=1}^\infty e_h(\tau)\,D_h\varphi
\end{equation}
is convergent \textcolor{black}{in $L^1(H,\nu)$} ($\tau$ is the unique point 
where the trajectories of $X(t)$ attain their minimum, see Hypothesis \ref{h2}).
Then the following identity holds
\begin{equation}
\label{e3}
\int_{\{g\ge r\}}\mathcal L\varphi\,\psi\,{\rm d} \mu=-\frac12\int_{\{g\ge r\}}\langle D\varphi, D\psi   \rangle\,{\rm d} \mu-\frac12\,\E\left[ \langle Dg,D\varphi   \rangle\,\psi |g=r\right]\rho(r).
\end{equation}

 \end{Proposition}
\begin{proof}
Let us start from identity \eqref{e6.3},
\begin{equation}
\label{e4}
\E\big[z(\tau)\varphi\big | g\!=\!r\big]\rho(r)=-\int_{\{g\ge r\}}[D\varphi\cdot z- \langle Q^{-1/2}x, Q^{-1/2}z  \rangle\, \varphi]\,{\rm d} \mu,\quad\forall\;r<0.
\end{equation}
Setting $z=e_h$, $x_h=\langle x,e_h\rangle$,  yields
\begin{equation}
\label{e5h}
\E\big[e_h(\tau_x)\,\varphi\big | g\!=\!r\big]\rho(r)=-\int_{\{g\ge r\}}( D_h\varphi- \tfrac{x_h}{\lambda_h}\, \varphi)\,{\rm d} \mu.
\end{equation}

Replacing $\varphi$ by $D_h\varphi\,\psi$  and consequently  
$D_h\varphi$ by $D^2_h\varphi\,\psi+D_h\varphi\,D_h\psi$, yields
\begin{equation}
\label{e6}
\E\big[e_h(\tau)\,D_h\varphi\,\psi\big | g\!=\!r\big]\rho(r)=-\int_{\{g\ge r\}}( D^2_h\varphi\,\psi- \tfrac{x_h}{\lambda_h}\, D_h\varphi\,\psi+D_h\varphi\,D_h\psi)\,{\rm d} \mu.
\end{equation}
Summing up on $h$ the conclusion follows.
  \end{proof}
\begin{Remark}
\label{r1}
\em

Let $\varphi$ be the weak solution of the Neumann problem. If {$ \varphi$ is sufficiently regular then}, {by comparing \eqref{e3} with  \eqref{e6b}} we deduce
\begin{equation}
\label{e8}
\E\left[ \langle Dg,D\varphi  \rangle\,\psi |g=r\right]=0,\quad\forall\;\psi\in L^2(H,\mu).
\end{equation}
 Identity  \eqref{e8}  can be   interpreted as a generalized Neumann condition.
 
Recall in fact that when  \textcolor{black}{ $g$ is a very regular function as for instance }$g=|x|^2$ then \eqref{e8} reduces to
\begin{equation}
\label{e88}
 \langle D\varphi,Dg\rangle=0\quad\mbox{\rm if}\;|x|^2=r,
 \end{equation}
 see \cite{BaDaTu09}, \cite{BaDaTu11} and \cite{DaLu15}. \textcolor{black}{In the present situation, where  $g$ is the infimum of a suitable process, we cannot hope that a boundary condition as \eqref{e88} is fulfilled but another condition should be guessed  with the help of \eqref{e8}.}

\end{Remark}

{\bf Aknowledgement}. We thank Lorenzo Zambotti for useful discussions.

\end{document}